\documentclass{amsart}

\usepackage{amsmath,amssymb}
\usepackage{mathtools, bm}
\usepackage{amsthm}
\usepackage[utf8]{inputenc}
\usepackage[T1]{fontenc}
\usepackage[english]{babel}
\usepackage{csquotes}
\usepackage{enumitem}
\usepackage[colorlinks=true]{hyperref}
\usepackage{etoolbox} 

\patchcmd{\subsubsection}{\itshape}{\bfseries}{}{}

\newcommand{\fundef}[5]{\begin{array}[t]{l|rcl}
#1: & #2 & \longrightarrow & #3 \\
    & #4 & \longmapsto & #5 \end{array}}
\newcommand{\ie}{i.e.\ }

\newcommand{\diff}{\mathop{}\mathopen{}\mathrm{d}}
\newcommand{\prob}{\mathrm{Prob}}
\newcommand{\Sp}{\mathbb{S}}
\newcommand{\rmG}{\mathrm{G}}

\newcommand{\R}{\mathbb{R}}

\newcommand{\N}{\mathbb{N}}

\newcommand{\calC}{\mathcal{C}}
\newcommand{\calM}{\mathcal{M}}
\DeclarePairedDelimiter{\abs}{\lvert}{\rvert}
\DeclarePairedDelimiter{\norm}{\lVert}{\rVert}
\newcommand{\indic}{\mathbf{1}}

\DeclareMathOperator*{\tang}{Tan}
\DeclareMathOperator*{\leb}{Leb}

\DeclareMathOperator{\haus}{\mathcal{H}}
\newcommand{\mres}{\mathop{\hbox{\vrule height 6pt width .5pt depth 0pt \vrule height .5pt width 4pt depth 0pt}}\nolimits}
\renewcommand{\epsilon}{\varepsilon}

\numberwithin{equation}{section}

\theoremstyle{plain}
\newtheorem{thm}{Theorem}[section]
\newtheorem{lem}[thm]{Lemma}
\newtheorem{prop}[thm]{Proposition}
\newtheorem{cor}[thm]{Corollary}

\theoremstyle{definition}

\theoremstyle{remark}
\newtheorem{rem}{Remark}[section]

\begin{document}
\title[Equivalence between branched transport models]{On the Lagrangian branched transport model and the equivalence with its Eulerian formulation}
\date\today
\author{Paul Pegon}
\address{Paul Pegon, Laboratoire de Mathématiques d'Orsay, Université Paris-Sud, 91405 Orsay cedex, France}
\email{paul.pegon@math.u-psud.fr}

\begin{abstract}
First we present two classical models of Branched Transport: the Lagrangian model introduced by Bernot, Caselles, Morel, Maddalena, Solimini \cite{bernot2009optimal, maddalena2003variational}, and the Eulerian model introduced by Xia \cite{xia2003optimal}. An emphasis is put on the Lagrangian model, for which we give a complete proof of existence of minimizers in a --hopefully-- simplified manner. We also treat in detail some $\sigma$-finiteness and rectifiability issues to yield rigorously the energy formula connecting the irrigation cost $\mathbf{I}_\alpha$ to the Gilbert Energy $\mathbf{E}_\alpha$. Our main purpose is to use this energy formula and exploit a Smirnov decomposition of vector flows, which was proved via the Dacorogna-Moser approach in \cite{santambrogio2014dacorogna}, to establish the equivalence between the Lagrangian and Eulerian models, as stated in Theorem \ref{eq_thm}.
\end{abstract}

\maketitle
\tableofcontents

\section*{Introduction}

Branched transport may be seen as an extension of the classical Monge-Kantorovich mass transportation problem (see \cite{OTAM} for a general reference). In this problem, we are given two probability measures $\mu,\nu$ representing the source and the target mass distributions and want to find a map $T$ which sends $\mu$ to $\nu$ in the most economical way. In the original problem from Monge, the cost of moving some mass $m$ along a distance $l$ is proportional to $m \times l$ and each particle moves independently to its destination along a straight line. For example, if one wants to transport a uniform mass on the segment $[-1,1]$ to a mass $2$ located at the point $y = (0,1)$, the mass will travel along straight lines departing from each point in $[-1,1]$ to $y$, hence there are infinitely many transport rays (in fact uncountably many). This is obviously not the most economical way to transport mass if we think for example of ground transportation networks; in this case you do not want to build infinitely many roads but you prefer to build a unique larger road, which ramifies near the source and the destination to collect and dispatch the goods. Hence in accurate models one should expect some branching structure to arise, which we actually observe in the optimal structures for the irrigation costs we deal with here.

This branching behavior may be seen in many supply-demand distribution systems such as road, pipeline or communication networks, but also natural systems like blood vessels, roots or river basins. It is usually due to \enquote{economy of scale} principles, which say, roughly speaking, that building something bigger will cost more, but proportionally less: once you have built the infrastructures, it does not cost much more to increase the traffic along the network. Thus it is in many cases more economically relevant to consider concave costs w.r.t. the mass, for instance costs of the form $m^\alpha \times l$ with $\alpha \in [0,1[$, which are strictly subadditive in $m$ and will force the mass to travel jointly as much as possible. Notice that to model such behavior either in Lagrangian or Eulerian frameworks, one needs to look at the paths actually followed by each particle, and this could not be done via transport maps $T$ or transports plans $\pi \in \Pi(\mu,\nu)$, which only describe how much mass goes from a location $x$ to another location $y$.

We shall present here the two main models of branched transport: the Eulerian model developed by Maddalena, Morel, Solimini \cite{maddalena2003variational}, later studied by Bernot, Caselles, Morel in \cite{bernot2009optimal}, and the Lagrangian model introduced by Xia in \cite{xia2003optimal} as an extension to a discrete model model proposed by Gilbert in \cite{Gil}. In Section \ref{lag_sec} we put an emphasis on the Lagrangian model and give a detailed and simple proof of existence of optimal irrigation plans without resorting to parameterizations as done in \cite{bernot2009optimal}, thus avoiding unnecessary technicalities and measurability issues. Also notice that, in this way, the Lagrangian branched transport model fits the general framework of dynamical transport problems with measures on curves, as in \cite{brenier2011modified} for Incompressible Euler and in \cite{brasco2010congested, conges2008} for traffic congestion. Section \ref{ener_sec} is devoted to a rigorous proof of the energy formula which connects the irrigation cost $\mathbf{I}_\alpha$ to the Gilbert Energy $\mathbf{E}_\alpha$, tackling some issues of $\sigma$-finiteness and rectifiability. In Section \ref{eul_sec} we give a brief description of the discrete model by Gilbert and the continuous extension proposed by Xia. The purpose of the article lies in Section \ref{equiv_sec} which establishes the equivalence between the two models, using the energy formula and a Smirnov decomposition (see \cite{smirnov1993decomposition}) obtained by Santambrogio in \cite{santambrogio2014dacorogna} via a Dacorogna-Moser approach, as stated in Theorem \ref{eq_thm}.

\section{The Lagrangian model: irrigation plans}\label{lag_sec}

\subsection{Notation and general framework}

Let $K$ be a compact subset of $\R^d$. We denote by $\Gamma(K)$ (or simply $\Gamma$) the space of $1$-Lipschitz curves in $K$ parameterized on $\R_+$, embedded with the topology of uniform convergence on compact sets. Recall that it is a compact metrizable space\footnote{One may use the distance $d(\gamma, \gamma') \coloneqq \sup_{n \in \N^\star} \frac{1}{n} \norm*{\gamma - \gamma'}_{L^\infty([0,n])}$.}.

\subsubsection*{Length and stopping time} If $\gamma \in \Gamma$, we define its \emph{stopping time} and its \emph{length} respectively by
\begin{align*}
T(\gamma) &= \inf\{t \geq 0 : \gamma\text{ is constant on } [t, +\infty[\},\\
L(\gamma) &= \int_0^{\infty} \abs*{\dot{\gamma}(t)} \diff t,
\end{align*}
which are valued in $[0, +\infty]$. Since curves are $1$-Lipschitz, $L(\gamma) \leq T(\gamma)$. Moreover, one may prove that $T$ and $L$ are both lower semicontinuous functions and, as such, are Borel. We denote by $\Gamma^1(K)$ the set of curves $\gamma$ with finite length, \ie those satisfying $L(\gamma) < \infty$.

\subsubsection*{Irrigation plans}  We call \emph{irrigation plan} any probability measure $\eta \in \prob({\Gamma})$ satisfying
\begin{equation}\label{fin_len}
\mathbf{L}(\eta) \coloneqq \int_\Gamma L(\gamma) \:\eta(\diff\gamma) < +\infty.
\end{equation}
Notice that any irrigation plan is concentrated on $\Gamma^1(K)$. If $\mu$ and $\nu$ are two probability measures on $K$, one says that $\eta \in \mathrm{IP}(K)$ irrigates $\nu$ from $\mu$ if one recovers the measures $\mu$ and $\nu$ by sending the mass of each curve respectively to its starting point and to its terminating point, which means that
\begin{align*}
(\pi_0)_\#\eta &= \mu,&
(\pi_\infty)_\#\eta &= \nu,
\end{align*}
where $\pi_0(\gamma) = \gamma(0)$, $\pi_\infty(\gamma) = \gamma(\infty) \coloneqq \lim_{t \to +\infty} \gamma(t)$ and $f_\#\eta$ denotes the push-forward of $\eta$ by $f$ whenever $f$ is a Borel map\footnote{Notice that $\lim_{t\to\infty} \gamma(t)$ exists if $\gamma \in \Gamma^1(K)$, and this is all we need since any irrigation plan is concentrated on $\Gamma^1(K)$.}. We denote by $\mathrm{IP}(\mu,\nu)$ the set of irrigation plans irrigating $\nu$ from $\mu$:
\[\mathrm{IP}(\mu,\nu) = \{ \eta \in \mathrm{IP}(K) : (\pi_0)_\#\eta = \mu, (\pi_\infty)_\#\eta = \nu\}.\]

\subsubsection*{Speed normalization} We say that a curve $\gamma \in \Gamma$ is \emph{parameterized by arc length} or \emph{normalized} if it has unit speed up until it stops, \ie $\abs*{\dot{\gamma}(t)} = 1$ for a.e. $t\in [0,T(\gamma)[$. If an irrigation plan $\eta \in \mathrm{IP}(K)$ is such that $\eta$-a.e. curve $\gamma$ is normalized, we say that $\eta$ is itself \emph{parameterized by arc length} or \emph{normalized}. Set $\mathrm{sn} : \gamma \mapsto \tilde{\gamma}$ the map which associates to each curve $\gamma \in \Gamma(K)$ its speed normalization\footnote{One may check that it is Borel.}. If $\eta \in \mathrm{IP}(K)$ is a general irrigation plan one may define its speed normalization as $\tilde{\eta} \coloneqq \mathrm{sn}_\# \eta$ and check that $(\pi_0)_\#\eta = (\pi_0)_\# \tilde{\eta}$ and $(\pi_\infty)_\# \eta = (\pi_\infty)_\# \tilde{\eta}$.

\subsubsection*{Multiplicity} Given an irrigation plan $\eta \in \mathrm{IP}(K)$, let us define the multiplicity $\theta_\eta : K \to [0,\infty]$ as
\[\theta_\eta(x) = \eta(\gamma \in \Gamma(K) : x \in \gamma),\]
which is also denoted by $\abs*{x}_\eta$. It represents the amount of mass which passes at $x$ through curves of $\eta$. We call \emph{domain} of $\eta$ the set $D_\eta$ of points with positive multiplicity (points that are really visited by $\eta$):
\[D_\eta \coloneqq \{ x \in K: \theta_\eta(x) > 0\}.\]

\subsubsection*{Simplicity} If $\gamma \in \Gamma$, we denote by
\[m(x,\gamma) = \#\{t \in [0,T(\gamma)] \cap \R_+ : \gamma(t) = x\} \in \N \cup \{\infty\}\]
the multiplicity of $x$ on $\gamma$, that is the number of times the curve $\gamma$ visits $x$. We call \emph{simple points} of $\gamma \in \Gamma$ those which are visited only once, \ie such that $m(x,\gamma) = 1$ and denote by $S_\gamma$ the set of such points. We say that $\gamma$ is \emph{simple} if $\gamma \setminus S_\gamma = \emptyset$ and \emph{essentially simple} if $\haus^1(\gamma \setminus S_\gamma) = 0$. As usual we extend these definitions to irrigation plans, saying that $\eta$ is \emph{simple} (resp. \emph{essentially simple}) if $\eta$-a.e. curve is simple (resp. essentially simple). Finally we set
\[m_\eta(x) \coloneqq \int_\Gamma m(x,\gamma) \eta(\diff \gamma)\]
which represents the mean number of times curves visit $x$. Notice that
\[\theta_\eta(x) = \int_\Gamma \indic_{x\in \gamma} \eta(\diff \gamma) \leq \int_\Gamma m(x,\gamma) \eta(\diff \gamma) \doteq m_\eta(x)\]
so that $m_\eta(x)$ is in a way the \enquote{full} multiplicity at $x$. 

\subsection{The Lagrangian irrigation problem}

\subsubsection*{Notation} If $\phi : K \to \R_+$ and $\psi : K \to \R^d$ are Borel functions on $K$, we set
\begin{align*}
\int_\gamma \phi(x) \abs*{\diff x} &\coloneqq \int_0^\infty \phi(\gamma(t)) \abs*{\dot{\gamma}(t)} \diff t \shortintertext{and}
\int_\gamma \psi(x) \cdot \diff x &\coloneqq \int_0^\infty \psi(\gamma(t)) \cdot \dot{\gamma}(t) \diff t
\end{align*}
provided $t\mapsto \abs*{\psi(\gamma(t))} \abs*{\dot{\gamma}(t)}$ is integrable.

\subsubsection*{Irrigation costs} For $\alpha \in [0,1]$ we consider the \emph{irrigation cost} $\mathbf{I}_\alpha : \mathrm{IP}(K) \to [0,\infty]$ defined by
\[\mathbf{I}_\alpha(\eta) \coloneqq \int_\Gamma \int_\gamma \abs*{x}_\eta^{\alpha -1} \abs*{\diff x} \eta(\diff \gamma),\]
with the conventions $0^{\alpha -1} = \infty$ if $\alpha <1$, $0^{\alpha -1} = 1$ otherwise, and $\infty \times 0 = 0$. If $\mu,\nu$ are two probability measures on $K$, we want to minimize the cost $\mathbf{I}_\alpha$ on the set of irrigation plans which send $\mu$ to $\nu$, which reads
\begin{equation}\label{lag_pb}\tag{$\text{LI}_\alpha$}
\min_{\eta \in \mathrm{IP}(\mu,\nu)}\quad \int_\Gamma \int_\gamma \abs*{x}_\eta^{\alpha -1} \abs*{\diff x}\eta(\diff \gamma).
\end{equation}
Notice that $\mathbf{I}_\alpha$ is invariant under speed normalization, thus we will often assume without loss of generality that irrigation plans are normalized.

The following result gives a sufficient condition for irrigability with finite cost, and may be found in \cite{bernot2009optimal}.

\begin{prop}[Irrigability]
If $\alpha > 1 - \frac{1}{d}$ then for every pair of measures $(\mu,\nu) \in \prob(K)$, there is an irrigation plan $\eta \in \mathrm{IP}(\mu,\nu)$ of finite $\alpha$-cost, \ie such that $\mathbf{I}_\alpha(\eta) < \infty$.
\end{prop}
\begin{rem}
One can actually show that the uniform measure on a unit cube can be irrigated from a unit Dirac mass if and only if $\alpha > 1 - \frac{1}{d}$, hence this condition is necessary for an arbitrary pair $(\mu,\nu)$ to be irrigable with finite $\alpha$-cost, provided for example that $K$ has non-empty interior.
\end{rem}

\subsection{Existence of minimizers}

In this section we prove the necessary lower semicontinuity and compactness results leading to the proof of existence of minimizers by the direct method of calculus of variations. We recall here that, unless stated otherwise, continuity properties on $\Gamma$ relate to the topology of uniform convergence on compact sets and on $\mathrm{IP}(K)$ to the weak-$\star$ topology in the duality with $\calC(K)$.

\subsubsection*{A tightness result} For $C > 0$ we define $\mathrm{IP}_C(K)$ as the set of irrigation plans $\eta$ on $K$ such that
\[\mathbf{T}(\eta) \coloneqq \int_\Gamma T(\gamma) \eta(\diff \gamma) \leq C\]
and $\mathrm{IP}_C(\mu,\nu) \coloneqq \mathrm{IP}_C(K) \cap \mathrm{IP}(\mu,\nu)$. Notice that for a normalized irrigation plan $\eta$ one has
\[\mathbf{T}(\eta) \doteq \int_\Gamma T(\gamma) \eta(\diff\gamma) = \int_\Gamma L(\gamma) \eta(\diff \gamma) \doteq \mathbf{L}(\eta).\]

\begin{rem}\label{IPC_compact_rem}
Since $\mathbf{T}$ is lower semicontinuous on $\prob(\Gamma)$ and $\mathbf{L} \leq \mathbf{T}$, $\mathrm{IP}_C(K)$ is a closed -- thus compact -- subset of $\prob(\Gamma)$.
\end{rem}

The following lemma results directly from Markov's inequality.

\begin{lem}[Tightness]\label{tight_lem}
Given $C > 0$, for all $\eta \in \mathrm{IP}_C(K)$ one has
\[\eta(\gamma : T(\gamma) > M) \leq \frac{C}{M}.\]
\end{lem}
This can be considered as a tightness result because it means that all irrigation plans $\eta \in \mathrm{IP}_C(K)$ are almost concentrated (uniformly) on the sets $\Gamma_M \coloneqq \{\gamma \in \Gamma : T(\gamma) \leq M\}$, which are compact for the uniform topology.

\subsubsection*{Continuity results} When $A$ is a closed subset of $K$, we set
\begin{align*}
[A] &= \{ \gamma \in \Gamma^1 : \gamma \cap A \neq \emptyset\},&
\abs*{A}_\eta &= \eta([A]),
\end{align*}
so that $\abs*{x}_\eta = \abs*{\{x\}}_\eta$. One may show that $[A]$ is Borel. Our first continuity result is that of $\abs*{\cdot}_\eta$ along decreasing sequences of closed sets.

\begin{lem}\label{decr_lem}
If $(A_n)_{n\in\N}$ is a decreasing sequence of closed sets and $A = \bigcap_n^\downarrow A_n$ then
\[\abs*{A}_\eta = \lim_{n\to\infty} \abs*{A_n}_\eta.\]
\end{lem}
\begin{proof}
Let us prove that $[A] = \bigcap_n^\downarrow [A_n]$. The inclusion $[A] \subseteq \bigcap_n [A_n]$ is clear since $[A] \subseteq[A_n]$ for all $n$. Let us take $\gamma \in \bigcap_n [A_n]$. Since belonging to some $[B]$ only depends on the trajectory of $\gamma$, we may assume that it is parameterized by arc length. Because $\gamma$ has finite length $L$, there is a sequence $(A_n)_n$ and a sequence $(t_n)_n \in [0,L]$ such that $\gamma(t_n) \in A_n$ for all $n$. One may extract a converging subsequence, still denoted $(t_n)_n$, such that $t_n \xrightarrow{n\to\infty} t \in [0,L]$. Since the $(A_n)$'s are decreasing closed sets, $\gamma(t)$ belongs to their intersection $A$, hence $\gamma \in [A]$. By the monotone convergence theorem
\[\lim_n \abs*{A_n}_\eta \doteq \lim_n \eta([A_n]) = \eta([A]) = \abs*{A}_\eta.\]
\end{proof}

\begin{prop}
For $C> 0$, the map
\[\fundef{\theta}{K \times \mathrm{IP}_C(K)}{[0,1]}{(x,\eta)}{\abs*{x}_\eta}\]
is upper semicontinuous.
\end{prop}
\begin{proof}
Given $x_n \to x$ and $\eta_n \rightharpoonup \eta$, take $\epsilon > 0$. If $n$ is large enough, $x_n \in \bar{B}(x, \epsilon)$, hence $\limsup_n \abs*{x_n}_{\eta_n} \leq \limsup_n \abs*{\bar{B}(x,\epsilon)}_{\eta_n}$. Besides, using Lemma \ref{tight_lem} one gets
\begin{align*}
\abs*{\bar{B}(x,\epsilon)}_{\eta_n} &\leq \eta_n(\{T> M\}) + \eta_n(\{T \leq M\}\cap [\bar{B}(x,\epsilon)])\\
&\leq C/M + \eta_n(A)
\end{align*}
where we set $A \coloneqq \{T\leq M\} \cap [\bar{B}(x,\epsilon)]$. It is easy to check $A$ is closed since $T$ is lsc and the ball is closed. Hence passing to the $\limsup$ in $n$ yields
\begin{equation}
\limsup_n \abs*{\bar{B}(x,\epsilon)}_{\eta_n} \leq C/M + \eta(A) \leq C/M + \abs*{\bar{B}(x,\epsilon)}_\eta.
\end{equation}
Taking $M\to\infty$, one gets
\[\limsup_n \abs*{x_n}_{\eta_n} \leq \limsup_n \abs*{\bar{B}(x,\epsilon)}_{\eta_n} \leq \abs*{\bar{B}(x,\epsilon)}_\eta,\]
then we pass to the limit $\epsilon \to 0$ using Lemma \ref{decr_lem}:
\[\limsup_n \abs*{x_n}_{\eta_n} \leq \lim_{\epsilon \to 0} \abs*{\bar{B}(x,\epsilon)}_\eta = \abs*{x}_\eta.\]
\end{proof}

For any $\eta \in \mathrm{IP}(K)$, we define the $\alpha$-cost of a curve $\gamma \in \Gamma$ w.r.t. $\eta$ by
\[Z_{\alpha,\eta}(\gamma) = \int_\gamma \theta_\eta^{\alpha -1}(x) \abs*{\diff x}\]
and we set $\fundef{Z_\alpha}{\Gamma\times\mathrm{IP(K)}}{\R}{(\gamma,\eta)}{Z_{\alpha,\eta}(\gamma)}$.

\begin{prop}\label{landscape_lsc}
For any $C>0$, the function $Z_\alpha$ is lower semicontinuous on $\Gamma \times \mathrm{IP}_C$.
\end{prop}
\begin{proof}
The case $\alpha = 1$ is clear since $Z_1(\gamma,\eta) = L(\gamma)$ hence we assume $\alpha < 1$. We know that the map $f : (x, \eta) \mapsto \abs*{x}_\eta^{\alpha-1}$ is lsc on $K \times \mathrm{IP_C}$ since $\theta : (x,\eta) \mapsto \abs*{x}_\eta$ is usc and $\alpha -1 < 0$. Now take $\gamma_n \to \gamma$ and $\eta_n \rightharpoonup \eta$, then since $T$ is lsc, for $\epsilon >0$ and $n$ large enough we have $T(\gamma) \leq T(\gamma_n) + \epsilon$, which implies that
\begin{equation}\label{landscape_lsc_eq1}
\int_0^{T(\gamma_n)} f(\gamma_n(t), \eta_n) \abs*{\dot{\gamma}_n(t)} \diff t \geq \int_0^{T(\gamma)-\epsilon} f(\gamma_n(t), \eta_n) \abs*{\dot{\gamma}_n(t)}\diff t.
\end{equation}
Suppose for a moment that $f$ is continuous on $K \times \mathrm{IP}_C$ which is a compact metric space, hence it is uniformly continuous. Since $\gamma_n$ converges uniformly to $\gamma$ on $[0,T(\gamma)-\epsilon)]$, the function $g_n : t \mapsto f(\gamma_n(t), \eta_n)$ converges uniformly to $g : t \mapsto f(\gamma(t), \eta)$ on $[0,T(\gamma) - \epsilon]$. Now we have to take care of the $\abs*{\dot{\gamma}_n(t)}$ factor. Since the sequence $(\dot{\gamma}_n)_n$ is bounded in $L^\infty([0,T(\gamma)-\epsilon])$ one may extract a subsequence $(\dot{\gamma}_{n_k})_k$ such that $\dot{\gamma}_{n_k} \xrightharpoonup{L^\infty} \dot{\gamma}$ and $\abs*{\dot{\gamma}_{n_k}} \xrightharpoonup{L^\infty} u$. It is a classical result that $\abs*{\dot{\gamma}(t)} \leq u(t)$ almost everywhere on $[0,T(\gamma) - \epsilon]$. Denoting by $\langle\cdot,\cdot\rangle$ the duality bracket $L^1-L^\infty$ on $[0,T(\gamma)-\epsilon]$, we have
\begin{align*}
\int_0^{T(\gamma)-\epsilon} f(\gamma_{n_k}(t), \eta_{n_k}) \abs*{\dot{\gamma}_{n_k}(t)} \diff t \doteq \langle g_{n_k}&, \abs*{\dot{\gamma}_{n_k}} \rangle\\
\xrightarrow{k \to \infty} \; \langle g&, u\rangle \geq \int_0^{T(\gamma) - \epsilon} f(\gamma(t), \eta) \abs*{\dot{\gamma}(t)} \diff t
\end{align*}
since $g_{n_k} \to g$ uniformly hence strongly in $L^1$. Prior to extracting the subsequence $(\dot{\gamma}_{n_k})_k$ we could have taken first a subsequence of $\dot{\gamma}_n$ such that the left hand-side converged to $\liminf_n \int_0^{T(\gamma) - \epsilon} f(\gamma_n(t),\eta_n) \abs*{\dot{\gamma}_n(t)} \diff t$. Thus we actually have
\[\liminf_n \int_0^{T(\gamma)-\epsilon} f(\gamma_n(t), \eta_n) \abs*{\dot{\gamma}_n(t)}\diff t \geq \int_0^{T(\gamma)-\epsilon} f(\gamma(t),\eta) \abs*{\dot{\gamma}(t)} \diff t.\]
Finally, we use this inequality together with \eqref{landscape_lsc_eq1} and pass to the limit $\epsilon \to 0$ thanks to the monotone convergence theorem, which yields
\begin{equation}\label{landscape_lsc_eq2}
\liminf_n  \int_0^{T(\gamma_n)} f(\gamma_n(t),\eta_n) \abs*{\dot{\gamma}_n(t)} \diff t
\geq \int_0^{T(\gamma)} f(\gamma(t),\eta) \abs*{\dot{\gamma}(t)}\diff t.
\end{equation}
In general $f$ is not continuous but only lsc, but \eqref{landscape_lsc_eq2} still holds for our function $f$ by considering an increasing sequence of continuous functions $f_k \uparrow f$, writing the inequality with $f_k$ and using the monotone convergence theorem as $k\to \infty$. We have therefore proven that
\[\liminf_n Z_\alpha(\gamma_n,\eta_n) \geq Z_\alpha(\gamma,\eta)\]
hence $Z_\alpha$ is lsc on $\Gamma \times \mathrm{IP}_C(K)$.
\end{proof}

Notice that our cost $\mathbf{I}_\alpha$ may be written as
\[\mathbf{I}_\alpha(\eta) = \int_\Gamma Z_\alpha(\gamma,\eta) \eta(\diff \gamma),\]
hence its lower semicontinuity on $\mathrm{IP}_C$ will be obtained as a corollary to the following lemma.

\begin{lem}\label{general_continuity}
Let $X$ be a closed subset of $\prob(\Gamma)$.
\begin{enumerate}[label=(\roman*)]
\item If $f : \Gamma \times X \to \R$ is continuous, then the functional $F : \eta \mapsto \int_\Gamma f(\gamma,\eta) \, \eta(d\gamma)$ is continuous on $X$.
\item If $f : \Gamma \times X \to [0,+\infty]$ is lsc, then $F : \eta \mapsto \int_\Gamma f(\gamma, \eta) \, \eta(d\gamma)$ is lsc on $X$.
\end{enumerate}
\end{lem}
\begin{proof}
Let us prove the first item. Take $\eta_n \rightharpoonup \eta$. Since $X$ is a compact metrizable space, so is $\Gamma\times X$ and by Heine's theorem $f$ is uniformly continuous on $\Gamma \times X$. Setting $g_n(\gamma) = f(\gamma, \eta_n)$ and $g(\gamma) = f(\gamma, \eta)$, this implies that $g_n \to g$ strongly in $\calC(\Gamma)$. Since $\eta_n \rightharpoonup \eta$ weakly-$\star$, we have $\int_\Gamma f(\gamma, \eta_n) \, \eta_n(d\gamma) = \langle \eta_n, g_n \rangle \to \langle \eta, g \rangle = \int_\Gamma f(\gamma, \eta) \, \eta(d\gamma)$ where $\langle \cdot, \cdot \rangle$ denotes the $\calC(\Gamma) - \calM(\Gamma)$ duality bracket.

The second item is a straightforward consequence of the fact that $f$ can be written as the increasing limit of continuous functions $f_k$, and of the monotone convergence theorem.
\end{proof}

\begin{cor}\label{lag_lsc}
For $\alpha\in[0,1]$, the functional $\mathbf{I}_\alpha$ is lower semincontinus on $\mathrm{IP}_C(K)$.
\end{cor}
\begin{proof}
We know by Remark \ref{IPC_compact_rem} that $\mathrm{IP}_C(K)$ is a closed subset of $\prob(\Gamma)$ which gives the result by the previous lemma applied to the function $f = Z_\alpha$ defined on $\mathrm{IP}_C(K) \times \Gamma$.
\end{proof}

Finally, we are left to investigate the continuity of the maps $\boldsymbol{\pi}_{\boldsymbol{0}} : \eta \mapsto (\pi_0)_\# \eta$ and $\boldsymbol{\pi}_{\boldsymbol{\infty}} : \eta \mapsto (\pi_\infty)_\#\eta$ on $\mathrm{IP}(K)$.
\begin{prop}\label{marg_cont}
The maps $\boldsymbol{\pi}_{\boldsymbol{0}}, \boldsymbol{\pi}_{\boldsymbol{\infty}} : \mathrm{IP}_C(K) \to \prob(K)$ are continuous\footnote{Recall that we always endow spaces of measures with their weak-$\star$ topology.}. In particular $\mathrm{IP}_C(\mu,\nu)$ is closed.
\end{prop}
\begin{proof}
Clearly the map $\pi_0 : \gamma \mapsto \gamma(0)$ is continuous on $\Gamma$, hence $\boldsymbol{\pi}_{\boldsymbol{0}}$ is continuous on $\mathrm{IP}(K)$. However $\pi_\infty : \gamma \mapsto \gamma(\infty)$ defined on $\Gamma^1$ is not necessarily continuous thus $\boldsymbol{\pi}_{\boldsymbol{\infty}}$ needs not be continuous on $\mathrm{IP}(K)$. Nevertheless, $\pi_\infty$ is continuous on all the sets $\Gamma_M$ for $M > 0$ and the tightness result of Lemma \ref{tight_lem} allows us to conclude. Indeed take $\eta_n \rightharpoonup \eta$ in $\mathrm{IP}_C(K)$. Take $\epsilon >0$ and $M$ large enough so that $\frac{C}{M} \leq \epsilon$. For any $\phi \in \calC(K)$ one has
\begin{align*}
\int_\Gamma \phi(\gamma(\infty)) \eta_n(\diff \gamma) &= \int_{\Gamma_M} \phi(\gamma(M)) \eta_n(\diff \gamma) + \int_{(\Gamma_M)^c} \phi(\gamma(\infty)) \eta_n(\diff \gamma)\shortintertext{and}
\int_\Gamma \phi(\gamma(\infty)) \eta(\diff \gamma) &= \int_{\Gamma_M} \phi(\gamma(M)) \eta(\diff \gamma) + \int_{(\Gamma_M)^c} \phi(\gamma(\infty)) \eta(\diff \gamma)
\end{align*}
thus
\begin{multline*}
\abs*{\int_\Gamma \phi(\gamma(\infty)) \eta_n(\diff \gamma) - \int_\Gamma \phi(\gamma(\infty)) \eta(\diff \gamma)}\\
\leq \abs*{\int_{\Gamma_M} \phi(\gamma(M)) \eta_n(\diff \gamma) - \int_{\Gamma_M} \phi(\gamma(M)) \eta(\diff \gamma)} + 2 \epsilon \norm*{\phi}_\infty.
\end{multline*}
We pass to the $\limsup_n$ using the continuity of $\pi_M : \gamma \mapsto \gamma(M)$ on $\Gamma$, then pass to the limit $\epsilon \to 0$ to yield
\[\int_K \phi(x) (\pi_\infty)_\#\eta_n(\diff x) \to \int_K \phi(x) (\pi_\infty)_\#\eta(\diff x)\]
which means that $\boldsymbol{\pi}_{\boldsymbol{\infty}}$ is continuous on $\mathrm{IP}_C(K)$.
\end{proof}

\subsubsection*{The existence theorem} We are now able to prove the existence theorem for the minimization problem \eqref{lag_pb}.

\begin{thm}
If $\mu,\nu$ are probability measures on $K$, there exists a minimizer $\eta$ of the problem
\begin{equation}\tag{$\text{LI}_\alpha$}
\min_{\eta \in \mathrm{IP}(\mu,\nu)}\quad \int_\Gamma \int_\gamma \abs*{x}_\eta^{\alpha -1} \abs*{\diff x}\eta(\diff \gamma).
\end{equation}
\end{thm}
\begin{proof}
We assume $\mathbf{I}_\alpha \not\equiv +\infty$, otherwise there is nothing to prove. Let us take a minimizing sequence $\eta_n$, which we may assume to be normalized. In particular $\mathbf{I}_\alpha(\eta_n) \leq C$ for some $C > 0$. Consequently
\[\mathbf{T}(\eta_n) = \mathbf{L}(\eta_n) \doteq \int_\Gamma \int_\gamma \abs*{\diff x} \eta_n(\diff\gamma) \leq \int_\Gamma \int_\gamma \abs*{x}_\eta^{\alpha -1} \abs*{\diff x} \eta_n(\diff \gamma) \doteq \mathbf{I}_\alpha(\eta_n) \leq C,\]
which implies that $\eta_n \in \mathrm{IP}_C(\mu,\nu)$. Thanks to Proposition \ref{marg_cont}, $\mathrm{IP}_C(\mu,\nu)$ is a closed subset of $\prob(\Gamma)$ which is compact (and metrizable) by Banach-Alaoglu's theorem, hence it is itself compact and we can extract a converging sequence $\eta_n \rightharpoonup \eta \in \mathrm{IP}_C(\mu,\nu)$ up to some renaming. By Corollary \ref{lag_lsc}, $\mathbf{I}_\alpha$ is lsc on $\mathrm{IP}_C(K)$, thus
\[\mathbf{I}_\alpha(\eta) \leq \liminf_n \mathbf{I}_\alpha(\eta_n) = \inf \eqref{lag_pb},\]
which shows that $\eta$ is a minimizer to the problem \eqref{lag_pb}.
\end{proof}

\section{The energy formula}\label{ener_sec}

The goal of this section is to establish the following formula
\begin{equation}\tag{EF}\label{ener_form}
\int_\Gamma \int_\gamma \abs*{x}_\eta^{\alpha -1} \abs*{\diff x}\eta(\diff \gamma) = \int_K \abs*{x}_\eta^\alpha \haus^1(\diff x),
\end{equation}
provided $\eta$ satisfies some hypotheses (namely essential simplicity and rectifiability). The term on the right-hand side is the so-called \emph{Gilbert Energy} denoted by $\mathbf{E}_\alpha(\eta)$. The proof relies solely on the correct use of Fubini-Tonelli's theorem, which requires $\sigma$-finiteness of measures. The next subsection is therefore devoted to the rectifiability of irrigation plans.

\subsection{Rectifiable irrigation plans}

\subsubsection*{Intensity and flow} We define the  \emph{intensity} $i_\eta \in \calM^+(K)$ and \emph{flow} $v_\eta \in \calM^d(K)$ of an irrigation plan $\eta \in \mathrm{IP}(\mu,\nu)$ by the formulas
\begin{align*}
\langle i_\eta, \phi \rangle &= \int_\Gamma \int_\gamma \phi(x) \abs*{\diff x} \:\eta(\diff\gamma),\\
\langle v_\eta, \psi \rangle &= \int_\Gamma \int_\gamma \psi(x) \cdot \diff x \:\eta(\diff\gamma),
\end{align*}
for all $\phi \in \calC(K), \psi \in \calC(K)^d$. The quantity $i_\eta(\diff x)$ represents the mean circulation at $x$ and $v_\eta(\diff x)$ the mean flow at $x$.

\subsubsection*{Concentration and rectifiability}
Let $A$ be a Borel subset of $K$. By definition of $i_\eta$, the following assertions are equivalent:
\begin{enumerate}[label=(\roman*)]
\item $i_\eta$ is concentrated on $A$, \ie $i_\eta(A^c) = 0$,
\item for $\eta$-a.e. $\gamma \in \Gamma$, $\gamma \subseteq A$ up to an $\haus^1$-null set, \ie $\haus^1(\gamma \setminus A) = 0$.
\end{enumerate}
In that case we say (with a slight abuse) that \emph{$\eta$ is concentrated on $A$}. An irrigation plan $\eta \in \mathrm{IP}(K)$ is termed \emph{$\sigma$-finite} if it is concentrated on a $\sigma$-finite set w.r.t. $\haus^1$ and \emph{rectifiable} if it is concentrated on a $1$-rectifiable set\footnote{A $1$-rectifiable set is the union of an $\haus^1$-null set with a countable union of Lipschitz curves.}.

The intensity $i_\eta$ has a simple expression when $\eta$ is $\sigma$-finite, as shown below.

\begin{prop}\label{intens_rectif}
If $\eta$ is an irrigation plan concentrated on a $\sigma$-finite set $A$, then $i_\eta = m_\eta \haus^1_{\mres A}$.
\end{prop}
\begin{proof}
For any Borel set $B$, one has
\begin{align*}
i_\eta(B) = i_\eta(A \cap B) &= \int_\Gamma \int_\gamma \indic_{A \cap B} \abs*{\diff x} \eta(\diff \gamma)\\
&= \int_\Gamma \int_B m(x,\gamma) \haus^1_{\mres A} (\diff x) \eta(\diff \gamma)\\
&= \int_B \int_\Gamma m(x,\gamma) \eta(\diff\gamma) \haus^1_{\mres A}(\diff x)\\
&= \int_B m_\eta(x) \haus^1_{\mres A}(\diff x).
\end{align*}
The equality on the second line follows from the coarea formula, the next from Fubini-Tonelli's theorem which holds because measures are $\sigma$-finite, and the last one from the definition of $m_\eta$.
\end{proof}

In order to prove that the domain $D_\eta \doteq \{ x \in K : \theta_\eta(x) > 0\}$ of an irrigation plan is $1$-rectifiable, we will need a few notions and non-trivial lemmas of geometric measure theory.

\subsubsection*{Density of a set} If $A$ is a subset of $\R^d$ we define the upper and lower $1$-density of $A$ at $x$ as
\begin{align*}
\overline{\Theta}(x,A) &= \limsup_{r \downarrow 0} \frac{\haus^1(A\cap B(x,r))}{2r},&
\underline{\Theta}(x,A) &= \liminf_{r \downarrow 0} \frac{\haus^1(A\cap B(x,r))}{2r}.
\end{align*}
When these quantities are equal, we call their common value $\Theta(x,A)$ the $1$-density of $A$ at $x$.

The first lemma we will need is proved in \cite[Chapter 8]{mattila1999geometry}.
\begin{lem}\label{haus_inn_reg}
If $B \subseteq \R^d$,
\[\haus^1(B) = \sup \{ \haus^1(K) : K \subseteq B \text{ compact such that } \haus^1(K) < \infty\}.\]
\end{lem}
The second is due to Besicovitch and may be obtained as a particular case of \cite[Theorem 17.6]{mattila1999geometry}.
\begin{lem}\label{rectif_lem}
Let $E$ be an $\haus^1$-measurable set such that $\haus^1(E) < \infty$. If its $1$-density $\Theta(x,E)$ exists and is equal to $1$ for $\haus^1$-a.e. $x$ in $E$ then it is $1$-rectifiable.
\end{lem}

Finally, we will need the following result which is included in \cite[Theorem 6.2]{mattila1999geometry}.

\begin{lem}\label{dens_ub}
If $E$ is a set such that $\haus^1(E) < \infty$, then the upper $1$-density $\overline{\Theta}(x,E)$ is less than $1$ for $\haus^1$-a.e. $x$ in $E$.
\end{lem}

\begin{prop}[Rectifiability of the domain]\label{rectif_dom}
If $\eta \in \mathrm{IP}(K)$ is an irrigation plan, its domain $D_\eta$ is $1$-rectifiable.
\end{prop}
\begin{proof}
First of all, since the domain does not change under normalization, we may assume that $\eta$ is parameterized by arc length. We have $D_\eta = \bigcup_{n>0} D^n$ where
\[D^n \coloneqq \left\{x : \theta_\eta(x) > \frac{1}{n}\right\}.\]
Let us show that that $\haus^1(D^n) < \infty$. By contradiction assume that for some $n$, $\haus^1(D^n) = \infty$, hence thanks to Lemma \ref{haus_inn_reg} one can find for $M >0$ as large as we want a compact subset $K' \subseteq D^n$ such that $M \leq \haus^1(K') < \infty$.
Since $\haus^1(K') < \infty$ one can use Fubini-Tonelli's theorem to get
\begin{align*}
\mathbf{L}(\eta) \doteq \int_\Gamma L(\gamma) \eta(\diff \gamma) \geq \int_\Gamma \haus^1(\gamma \cap K') \eta(\diff \gamma) &= \int_\Gamma \int_{K'} \indic_{x \in \gamma} \haus^1(\diff x) \eta(\diff \gamma)\\
&= \int_{K'} \int_\Gamma \indic_{x\in \gamma} \eta(\diff\gamma) \haus^1(\diff x)\\
&= \int_{K'} \theta_\eta(x) \haus^1(\diff x) > \frac{M}{n}.
\end{align*}
The inequality $\mathbf{L}(\eta) > \frac{M}{n}$ must be true for all $M > 0$, \ie $\mathbf{L}(\eta) = \infty$, which contradicts the definition of an irrigation plan, hence $\haus^1(D^n) < \infty$.

Now we shall prove that $\Theta(x,D^n) = 1$ a.e. on $D^n$. Since $D^n$ has finite $\haus^1$-measure, we already know $\overline{\Theta}(x,D^n) \leq 1$ for $\haus^1$-a.e. $x \in D^n$ by Lemma \ref{dens_ub}, thus it remains only to prove $\underline{\Theta}(x,D^n) \geq 1$. If $A$ is a Borel subset of $\R$ we denote $\leb(A)$ the set of Lebesgue points of $A$, which are points $t$ such that
\[\lim_{r\downarrow 0} \frac{\abs*{A \cap [t-r,t+r]}}{2r} = 1,\]
where $\abs*{X}$ denotes the Lebesgue measure of $X \subseteq \R$. Recall that by Lebesgue's theorem we have $\abs*{A \setminus \leb(A)} = 0$. For any $\gamma \in \Gamma^1$, we set
\begin{align*}
A_\gamma &= \left\{t : t \in \leb\left(s : \abs*{\gamma(s)}_\eta > \frac{1}{n}\right) \text{ for all $n$ s.t. } \abs*{\gamma(t)}_\eta > \frac{1}{n}\right\},\\
B_\gamma &=  \left\{t \in ]0,T(\gamma)[\; : \dot{\gamma}(t) \text{ exists}\right\},\\
D_\gamma &= \gamma(A_\gamma \cap B_\gamma).
\end{align*}
Notice that $\abs*{[0,T(\gamma)] \setminus (A_\gamma \cup B_\gamma)} = 0$ hence $\haus^1(\gamma \setminus D_\gamma) = 0$ since $\gamma$ is Lipschitz. Finally we set
\[D' = \bigcup_{\gamma \in \Gamma^1} D_\gamma.\]
Let us check that $\haus^1(D_\eta \setminus D') = 0$. We have
\begin{align*}
\int_{D_\eta \setminus D'} \theta_\eta(x) \haus^1(\diff x) &= \int_{D_\eta \setminus D'} \int_\Gamma \indic_{x\in \gamma} \eta(\diff \gamma) \haus^1(\diff x)\\
&= \int_\Gamma \int_{D_\eta \setminus D'} \indic_{x\in \gamma} \haus^1(\diff x) \eta(\diff \gamma)\\
&= \int_{\Gamma^1} \haus^1(D_\eta \cap \gamma \setminus D') \eta(\diff \gamma)\\
&= 0.
\end{align*}
The use of Fubini-Tonelli's theorem is justified since $D_\eta = \bigcup_n D^n$ is $\sigma$-finite and the last equality follows from $\haus^1(D_\eta \cap \gamma \setminus D') \leq \haus^1(\gamma \setminus D_\gamma) = 0$. This implies $\haus^1(D_\eta \setminus D') = 0$ since $\theta_\eta > 0$ on $D_\eta$. Now take any $x \in D^n \cap D'$. By construction of $D'$ there is a curve $\gamma \in \Gamma^1$ and a $t \in A_\gamma \cap B_\gamma$ such that $x = \gamma(t)$, which implies that
\begin{align}\label{rectif_dom_eq1}
\frac{\abs*{s \in [t-r, t+r] : \gamma(s) \in D^n}}{2r} &\xrightarrow{r \downarrow 0} 1\shortintertext{and}
\frac{\haus^1(\gamma([t-r,t+r]) \setminus D^n)}{2r} &\xrightarrow{r \downarrow 0} 0.\label{rectif_dom_eq2}
\end{align}
It follows from \eqref{rectif_dom_eq2} and the fact that $\gamma([t-r,t+r]) \subseteq \bar{B}(x,r)$ because $\gamma$ is $1$-Lipschitz that
\[\underline{\Theta}(x,D^n) \doteq \liminf_{r\downarrow 0} \frac{\haus^1(B(x,r) \cap D^n)}{2r} \geq \liminf_{r\downarrow 0} \frac{\haus^1(\gamma([t-r,t+r]))}{2r}.\]
But $\gamma$ has a derivative $e$ at $t$ which has unit norm. Moreover the $\haus^1$-measure of $\gamma([t-r,t+r])$, which is a compact connected set, is larger than the distance between $\gamma(t-r)$ and $\gamma(t+r)$, and since $\gamma(t \pm r) = x \pm r e + o(r)$ one has
\[\haus^1(\gamma([t-r,t+r])) \geq \abs*{\gamma(t+r) - \gamma(t-r)} = 2r + o(r),\]
which yields $\underline{\Theta}(x,D^n) \geq 1$.
This proves that $\Theta(x,D^n)$ exists and is equal to $1$ for $\haus^1$-a.e. $x \in D^n$ hence $D^n$ is $1$-rectifiable by Lemma \ref{rectif_lem} and $D_\eta = \bigcup_n D^n$ as well.
\end{proof}

At this stage we have shown that the domain of any irrigation plan is rectifiable, yet this does not mean that any irrigation plan is rectifiable (this is obviously not the case) since $\eta$ needs not be concentrated on $D_\eta$. However, it is essentially the only candidate rectifiable set (or even candidate $\sigma$-finite set) on which $\eta$ could be concentrated, as stated below.

\begin{cor}\label{rectif_equiv}
Given $\eta \in \mathrm{IP}(K)$ an irrigation plan, the following assertions are equivalent:
\begin{enumerate}[label=(\roman*)]
\item\label{rectif_equiv1} $\eta$ is concentrated on $D_\eta$,
\item\label{rectif_equiv2} $\eta$ is rectifiable,
\item\label{rectif_equiv3} $\eta$ is $\sigma$-finite.
\end{enumerate}
\end{cor}
\begin{proof}
It is enough to prove $\ref{rectif_equiv3} \Rightarrow \ref{rectif_equiv1}$ by the previous proposition. If $\eta$ is concentrated on a $\sigma$-finite set $A$, we know by Proposition \ref{intens_rectif} that $i_\eta = m_\eta \haus^1_{\mres A}$. Therefore $i_\eta$ is also concentrated on $\{x : m_\eta(x) > 0\} = D_\eta$.
\end{proof}
\begin{rem}\label{intens_rectif_rem}
From this and Proposition \ref{intens_rectif} we get that $i_\eta = m_\eta \haus^1$ if $\eta$ is rectifiable.
\end{rem}

The most important consequence of Proposition \ref{rectif_dom} is the following rectifiability result.

\begin{thm}[Rectifiability]\label{rectif_thm}
If $\eta$ has finite $\alpha$-cost with $\alpha \in [0,1[$, it is rectifiable.
\end{thm}
\begin{proof}
Because of the previous statement, we need only show that $\eta$ is concentrated on $D_\eta$. We have $\mathbf{I}_\alpha(\eta) \doteq \int_\Gamma \int_\gamma \abs*{x}_\eta^{\alpha -1} \abs*{\diff x} \eta(\diff \gamma) < \infty$ hence for $\eta$-almost every curve $\gamma$, for $\haus^1$-almost every $x$ in $\gamma$, $\abs*{x}_\eta^{\alpha -1} < \infty$, which implies that $\abs*{x}_\eta > 0$ \ie $x\in D_\eta$. By definition, it means that $\eta$ is concentrated on $D_\eta$.
\end{proof}

\subsection{Proof of the energy formula}

We define the Gilbert Energy $\mathbf{E}_\alpha : \mathrm{IP}(K) \to [0,\infty]$ as
\[\mathbf{E}_\alpha(\eta) =
\begin{dcases*}
\int_K \theta_\eta^\alpha(x) \haus^1(\diff x) & if $\eta$ is rectifiable,\\
+\infty& otherwise.
\end{dcases*}\]
and a variant $\bar{\mathbf{E}}_\alpha : \mathrm{IP}(K) \to [0,\infty]$ (kind of a \enquote{full} energy)
\[\bar{\mathbf{E}}_\alpha(\eta) =
\begin{dcases*}
\int_K \theta_\eta^{\alpha-1}(x) m_\eta(x) \haus^1(\diff x) & if $\eta$ is rectifiable,\\
+\infty& otherwise.
\end{dcases*}\]
Assuming $\alpha \in [0,1[$, we would like to establish the energy formula
\begin{equation}\tag{EF}
\mathbf{I}_\alpha(\eta) = \mathbf{E}_\alpha(\eta).
\end{equation}
This does not hold in general. Actually we are going to show that $\mathbf{I}_\alpha(\eta) = \bar{\mathbf{E}}_\alpha(\eta)$ for all irrigation plan $\eta \in \mathrm{IP}(K)$ and that $\bar{\mathbf{E}}_\alpha(\eta) = \mathbf{E}_\alpha(\eta)$ provided $\eta$ is essentially simple.

\begin{thm}[Energy formula]\label{thm:ener_form}
Assuming $\alpha \in [0,1[$, the following formula holds:
\begin{equation}\tag{EF'}\label{ener_form_bis}
\mathbf{I}_\alpha(\eta) = \bar{\mathbf{E}}_\alpha(\eta).
\end{equation}
Moreover, if $\eta$ is essentially simple this rewrites
\begin{equation}\tag{EF}
\mathbf{I}_\alpha(\eta) = \mathbf{E}_\alpha(\eta).
\end{equation}
\end{thm}
\begin{proof}
By Theorem \ref{rectif_thm}, if $\eta$ is not rectifiable then $\mathbf{I}_\alpha(\eta) = \mathbf{E}_\alpha(\eta) = \bar{\mathbf{E}}_\alpha(\eta) = \infty$ and the result is clear. Now we assume that $\eta$ is rectifiable, which means that it is concentrated on the rectifiable set $D_\eta$, according to Theorem \ref{rectif_thm} and Corollary \ref{rectif_equiv}. Notice that by the coarea formula we have
\[\int_\Gamma \int_\gamma \abs*{x}_\eta^{\alpha -1} \abs*{\diff x} \eta(\diff\gamma)= \int_\Gamma \int_K \abs*{x}_\eta^{\alpha -1} m(x,\gamma) \haus^1(\diff x) \eta(\diff \gamma),\]
thus the goal is to reverse the order of integration. Here Fubini-Tonelli's theorem applies because $\eta$ is concentrated on its domain, which is rectifiable, which yields
\begin{align*}
\mathbf{I}_\alpha(\eta) &= \int_\Gamma \int_{D_\eta} \abs*{x}_\eta^{\alpha -1} m(x,\gamma) \haus^1(\diff x) \eta(\diff \gamma)\\
 &= \int_{D_\eta} \abs*{x}_\eta^{\alpha -1} m_\eta(x) \haus^1(\diff x)\\
 &= \int_K \theta_\eta^{\alpha -1}(x) m_\eta(x) \haus^1(\diff x) = \bar{\mathbf{E}}_\alpha(\eta).
\end{align*}
and \eqref{ener_form_bis} holds. Now if $\eta$ is essentially simple then in all the previous calculations $m_\eta(x) = \theta_\eta(x)$ so that
\begin{align*}
\mathbf{I}_\alpha(\eta) &= \int_{D_\eta} \int_\Gamma \theta_\eta^{\alpha -1}(x) \theta_\gamma(x) \eta(\diff \gamma) \haus^1(\diff x)\\
 &= \int_K \theta_\eta^\alpha(x) \haus^1(\diff x) = \mathbf{E}_\alpha(\eta),
\end{align*}
thus getting \eqref{ener_form}.
\end{proof}
\begin{rem}
Actually, the proof shows that the equality $\mathbf{I}_\alpha(\eta) = \bar{\mathbf{E}}_\alpha(\eta)$ (and $\mathbf{I}_\alpha(\eta) = \mathbf{E}_\alpha(\eta)$ if $\eta$ is essentially simple) holds also for $\alpha = 1$ provided $\eta$ is rectifiable. However, one may find $\eta$ non-rectifiable such that $\mathbf{I}_1(\eta) \in ]0,\infty[$ while $\haus^1(D_\eta) = 0$. In that case one has $0 = \int_K m_\eta(x) \haus^1(\diff x) < \mathbf{I}_1(\eta) < \bar{\mathbf{E}}_1(\eta) = \infty$. Notice also that $0 = \int_K \theta_\eta^{\alpha -1}(x) m_\eta(x) \haus^1(\diff x) < \mathbf{I}_\alpha(\eta) = \infty$ for $\alpha \in [0,1[$, which explains why we imposed $\mathbf{E}_\alpha(\eta) = \bar{\mathbf{E}}_\alpha(\eta) = \infty$ if $\eta$ is not rectifiable.
\end{rem}

\subsection{Optimal irrigation plans are simple}

In this section we shall prove that optimal irrigation plans are necessary simple using the energy formula.

\subsubsection*{\enquote{Reduced} intensity} We associate to any irrigation plan $\eta \in \mathrm{IP}(K)$ a \enquote{reduced} intensity $j_\eta$ by
\[\langle j_\eta, \phi \rangle = \int_\Gamma \int_\gamma \phi(x) \haus^1(\diff x) \eta(\diff\gamma),\]
for all $\phi \in \calC(K)$. It is a positive finite measure, since the total mass is
\[\norm*{j_\eta} = \int_\Gamma \haus^1(\gamma) \eta(\diff \gamma) \leq \int_\Gamma L(\gamma) \eta(\diff \gamma) = \mathbf{L}(\eta) < \infty.\]
\begin{rem}\label{reduc_intens_rem}
Notice that if $A$ is a Borel set, $j_\eta(A) = 0 \Leftrightarrow i_\eta(A) = 0$ hence by definition $\eta$ is rectifiable if and only if $j_\eta$ is concentrated on a rectifiable set, in which case it is concentrated on the rectifiable domain $D_\eta$ and one has $j_\eta = \theta_\eta \haus^1$ using Fubini-Tonelli's theorem. 
\end{rem}

\begin{lem}[Simple replacement]\label{simple_rep}
Let $\eta \in \mathrm{IP}(\mu,\nu)$ be an irrigation plan. Consider the minimization problem
\begin{equation}\label{lf_pb}\tag{$\text{LEN}_\eta$}
\min\quad \quad \mathbf{L}(\zeta) \;  : \; j_\zeta \leq j_\eta \text{ and } \zeta \in \mathrm{IP}(\mu,\nu).
\end{equation}
Then
\begin{enumerate}[label=(\roman*)]
\item\label{simple_rep1} this problem admits minimizers which are all simple,
\item\label{simple_rep2} if $\eta$ is rectifiable, all minimizers $\zeta$ are also rectifiable and $j_\zeta \leq j_\eta$ rewrites
\begin{equation}\label{theta_ineq}
\theta_\zeta \leq \theta_\eta \quad \haus^1\text{-almost everywhere}.
\end{equation}
\end{enumerate}
Any minimizer of \eqref{lf_pb} is called a \emph{simple replacement} of $\eta$.
\end{lem}
\begin{proof}
Let us call $m$ the infimum of \eqref{lf_pb} and show that it admits a minimizer. Take a minimizing sequence $(\zeta_n)_n$ such that every $\zeta_n$ is normalized, in particuliar $\zeta_n \in \mathrm{IP}_C(K)$ for some $C > 0$. Up to extraction  we have convergence $\zeta_n \rightharpoonup \zeta$, and since $\mathrm{IP}_C(\mu,\nu)$ is closed by Proposition \ref{marg_cont}, $\zeta \in \mathrm{IP}_C(\mu,\nu)$. Moreover $\mathbf{L}(\zeta) = m$ by lower semicontinuity of $\mathbf{L} \doteq \mathbf{I}_1$ on $\mathrm{IP}_C(K)$, which we proved in Corollary \ref{lag_lsc}. Now in order to show that $\zeta$ is a solution of \eqref{lf_pb} we only have to check the last constraint $j_\zeta \leq j_\eta$. Take any open set $O$. One has
\[j_{\zeta_n}(O) = \int_\Gamma \haus^1(\gamma \cap O) \zeta_n(\diff \gamma).\]
By a generalization of Golab's Theorem (see \cite{Brancolini2010}), the following holds
\[\haus^1(\gamma\cap O) \leq \liminf_n \haus^1(\gamma_n \cap O)\]
if $\gamma_n \to \gamma$ uniformly on compact sets, which means that $\gamma \mapsto \haus^1(\gamma \cap O)$ is lower semicontinuous on $\Gamma$. Consequently $\zeta \mapsto \int_\Gamma \haus^1(\gamma \cap O) \zeta(\diff \gamma)$ is lower semicontinuous and one gets
\[j_\zeta(O) \doteq \int_\Gamma \haus^1(\gamma \cap O) \zeta(\diff \gamma) \leq \liminf_n \int_\Gamma \haus^1(\gamma \cap O) \zeta_n(\diff \gamma) \leq j_\eta(O)\]
for all open set $O$. This implies that $j_\zeta \leq j_\eta$ by regularity of finite measures hence $\zeta$ is a minimizer of \eqref{lf_pb}.

Let us check that any minimizer $\zeta$ is simple. By contradiction, if it was not simple there would be a set $\Gamma' \subseteq \Gamma$ such that $\zeta(\Gamma') > 0$ and every $\gamma \in \Gamma'$ has a loop. One may define a Borel map $r : \gamma \mapsto r(\gamma)$ which removes from $\gamma \in \Gamma'$ the loop with maximal length (the first one in case there are several), and is identical on $\Gamma \setminus \Gamma'$. Then set $\bar{\zeta} \coloneqq r(\zeta)$. Obviously one has $\mathbf{L}(\bar{\zeta}) < \mathbf{L}(\zeta)$, $\bar{\zeta} \in \mathrm{IP}(\mu,\nu)$ and $j_{\bar{\zeta}} \leq j_\zeta$, which contradicts the optimality of $\zeta$ in \eqref{lf_pb}.

Finally, suppose $\eta$ is rectifiable and take $\zeta$ a minimizer of our problem. According to Remark \ref{reduc_intens_rem}, the inequality $j_\zeta \leq j_\eta$ implies that $\zeta$ is rectifiable and $j_\eta = \theta_\eta \haus^1, j_\zeta = \theta_\zeta \haus^1$, which yields \ref{simple_rep2}.
\end{proof}

\begin{prop}
Given $\alpha \in [0,1]$, if $\eta \in \mathrm{IP}(\mu,\nu)$ is optimal with finite $\alpha$-cost, then it is simple.
\end{prop}
\begin{proof}
The case $\alpha = 1$ is straightforward from Lemma \ref{simple_rep} since $\mathbf{L} = \mathbf{I}_1$. Now we assume that $\alpha < 1$ and take $\eta$ optimal, in which case the finiteness of the $\alpha$-cost implies the rectifiability of $\eta$ by Theorem \ref{rectif_thm}. We need only show that $\eta$ is a minimizer of \eqref{lf_pb}. Take $\tilde{\eta}$ a simple replacement of $\eta$. Then since $\eta,\tilde{\eta}$ are rectifiable and $\theta_{\tilde{\eta}} \leq \theta_\eta$ $\haus^1$-a.e., one has
\[\mathbf{I}_\alpha(\eta) = \int_K \theta_\eta^{\alpha -1} m_\eta \diff\haus^1 \geq \int_K \theta_\eta^\alpha \diff \haus^1 \geq \int_K \theta_{\tilde{\eta}}^\alpha \diff\haus^1 = \mathbf{I}_\alpha(\tilde{\eta})\]
Since $\eta$ is optimal we have equality everywhere, which means that $m_\eta = \theta_\eta = \theta_{\tilde{\eta}} = m_{\tilde{\eta}}$ $\haus^1$-a.e.. Consequently
\[\mathbf{L}(\eta) = \int_K m_\eta(x) \haus^1(\diff x) = \int_K m_{\tilde{\eta}}(x) \haus^1(\diff x) = \mathbf{L}(\tilde{\eta})\]
hence $\eta$ minimizes \eqref{lf_pb} and is as such simple by Lemma \ref{simple_rep}.
\end{proof}

\section{The Eulerian model: irrigation flows}\label{eul_sec}

In this section we present the Eulerian model of branched transport, which was introduced by Xia in \cite{xia2003optimal} as a continuous extension to a discrete model proposed by Gilbert in \cite{Gil}.

\subsection{The discrete model}

\subsubsection*{Oriented Graph} An \emph{oriented graph} in $K$ is a pair $G = (E, w)$ where $E = E(G)$ is a set of oriented segments $(e_1,\ldots,e_n)$ called \emph{edges} and $w : E(G) \to]0,\infty[$ is a function which gives a weight to any edge. An oriented segment $e$ simply consists of an ordered pair of points $(e^-,e^+)$ in $K$ which we call \emph{starting and ending point} of $e$. We denote by $\abs*{e} \coloneqq \abs*{e^+-e^-}$ its length, by $\hat{e} \coloneqq \frac{e^+-e^-}{\abs*{e^+-e^-}} \in \Sp^{d-1}$ its \emph{orientation} provided $e^+ \neq e^-$, and set $\mathrm{G}(K)$ to be the set of oriented graphs on $K$.

\subsubsection*{Irrigation Graphs}
Given two atomic probability measures $\mu = \sum_i \alpha_i \delta_{x_i}$ and $\nu = \sum_i \beta_i \delta_{y_i}$, we say that $G$ irrigates $\nu$ from $\mu$ if it satisfies the so-called Kirchhoff condition, well-known for electric circuits:
\begin{center}
incoming mass at $v$ = outcoming mass at $v$,
\end{center}
for all vertex $v$ of the graph. By \enquote{incoming mass} we mean the total weight of edges with terminating point $v$, increased by $a_i$ if $v = x_i$, and by \enquote{outcoming mass} we mean the total weight of edges with starting point $v$, increased by $b_j$ if $v = y_j$. Indeed $\mu$ and $\nu$ are seen as mass being respectively pushed in and out of the graph. The set of graphs irrigating $\nu$ from $\mu$ is denoted by $\rmG(\mu,\nu)$.

\subsubsection*{Discrete Irrigation Problem}
With a slight abuse, we define the $\alpha$-cost of a graph as
\[\mathbf{E}_\alpha(G) = \sum_{e\in E(G)} w(e)^\alpha \abs*{e},\]
which means that the cost of moving a mass $m$ along a segment of length $l$ is $m^\alpha \cdot l$. Given $\mu, \nu$ two atomic probability measures, we want to minimize the cost of irrigation among all graphs sending $\mu$ to $\nu$, which reads
\begin{equation}\tag{$\text{DI}_\alpha$}
\min_{G\in\mathrm{G}(\mu,\nu)} \quad \mathbf{E}_\alpha(G).
\end{equation}

\subsection{The continuous model}

From now on we assume that $\alpha \in [0,1[$.

\subsubsection*{Irrigation Flow} We call \emph{irrigation flow} on $K$ any vector measure $v \in \calM^d(K)$ such that $\nabla \cdot v \in \calM(K)$, where $\nabla \cdot v$ is the divergence of $v$ in the sense of distribution. We denote by $\mathrm{IF}(K)$ the set of irrigation flows.

\begin{rem}
These objects have several names. They are called $1$-dimensional normal currents in the terminology of Geometric Measure Theory, and are called traffic paths by Xia in \cite{xia2003optimal}.
\end{rem}

If $E \subseteq K$ is an $\haus^1$-measurable set, $\tau : E \to \Sp^{d-1}$ is $\haus^1$-measurable and $\theta : E \to \R_+$ is $\haus^1$-integrable, we define the vector measure $[E, \tau, \theta] \in \calM^d(K)$ by
\[<[E, \tau, \theta], \psi> = \int_E \theta(x) \psi(x) \cdot \tau(x) \haus^1(\diff x),\]
for all $\psi \in \calC(K, \R^d)$. In other terms $[E, \tau, \theta] \doteq \theta \tau \haus^1_{\mres E}$.
\subsubsection*{Rectifiable irrigation flow} Recall that if $E$ is a $1$-rectifiable set, at $\haus^1$-a.e $x\in E$ there is an \emph{approximate tangent line} (see \cite[Chapter 17]{mattila1999geometry}) denoted by $\tang(x,E)$. An irrigation flow of the form $v = [E,\tau,\theta]$ where $E$ is $1$-rectifiable and $\tau(x) \in \tang(x,E)$ for $\haus^1$-a.e. $x \in E$ is termed \emph{rectifiable}.

\subsubsection*{From discrete to continuous} Consider a graph $G \in \mathrm{G}(K)$. One can define the vector measure $v_G$ by
\[v_G \coloneqq \sum_{e\in E(G)} [e, \hat{e}, w(e)].\]
One can check that $\nabla \cdot v_G = \sum_{e\in E(G)} w(e)(\delta_{e^-} - \delta_{e^+}) \in \calM(K)$ and that it is a rectifiable irrigation flow on $K$. Also, both the cost $\mathbf{E}_\alpha$ and the constraint $G \in \mathrm{G}(\mu,\nu)$ can be expressed solely in terms of $v_G$. Indeed if we identify $v_G$ with its $\haus^1$-density, one has
\[\mathbf{E}_\alpha(G) \doteq \sum_{e\in E(G)} w(e)^\alpha \abs*{e} = \int_K \abs*{v_G(x)}^\alpha \haus^1(\diff x).\]
And the Kirchhoff condition is expressed in terms of the divergence $\nabla\cdot v_G$:

\begin{prop}
If $G$ is a graph, $G\in \rmG(\mu,\nu)$ if and only if $\nabla \cdot v_G = \mu - \nu$.
\end{prop}

This leads to defining the following cost on $\mathrm{IF}(K)$:
\[\mathbf{M}_\alpha(v) = \begin{dcases*}
\int \abs*{v(x)}^\alpha \haus^1(dx) & if $v$ is rectifiable,\\
+\infty& otherwise,
\end{dcases*}\]
which is called the $\alpha$-mass of $v$. Actually, Xia gave a different definition of $\mathbf{M}_\alpha$ in \cite{xia2003optimal}, as a relaxation of the $\mathbf{E}_\alpha$ functional:
\[\mathbf{M}_\alpha(v) = \inf_{v_{G_n} \xrightharpoonup{\mathbf{N}} v} \liminf_n \mathbf{E}_\alpha(G_n),\]
where $v_n \xrightharpoonup{\mathbf{N}} v$ means that $v_n \rightharpoonup v$ and $\nabla \cdot v_n \rightharpoonup \nabla \cdot v$ weakly-$\star$ as measures on $K$. These definitions coincide on $\mathrm{IF}(K)$ as shown in \cite{xia2004interior}, and $\mathbf{M}_\alpha(v) = \int \abs{v(x)}^\alpha \haus^1(\diff x)$ as soon as $v$ has an $\haus^1$-density. Finally, we say that $v$ sends $\mu$ to $\nu$ if $\nabla \cdot v = \mu - \nu$ and denote by $\mathrm{IF}(\mu,\nu)$ the set of such irrigation flows.

\subsubsection*{Eulerian irrigation problem} We are now able to formulate an Eulerian irrigation problem in a continuous setting. Given two probability measures $\mu,\nu \in \prob(K)$, we want to find an irrigation flow $v$ sending $\mu$ to $\nu$ which has minimal $\alpha$-mass. This reads
\begin{equation}\tag{$\text{EI}_\alpha$}\label{eul_pb}
\min_{v\in\mathrm{IF}(\mu,\nu)} \quad \mathbf{M}_\alpha(v).
\end{equation}
Xia proved the following theorem in \cite{xia2003optimal}.

\begin{thm}[Existence of minimizers]
For any $\mu,\nu \in \prob(K)$, there is a minimizer $v$ to the problem \eqref{eul_pb}. Moreover if $1 - \frac{1}{d} < \alpha < 1$ the minimum is always finite.
\end{thm}

\section{Equivalence between models}\label{equiv_sec}

In this section we show that the Lagrangian and Eulerian irrigation problems are equivalent, in the sense that they have same minimal value, and one can build minimizers of one problem from minimizers of the other. In this section we assume $1 - \frac{1}{d} < \alpha < 1$.

\subsection{From Lagrangian to Eulerian}

Recall that we have associated to any irrigation plan $\eta \in \mathrm{IP}(K)$ an intensity $i_\eta \in \calM^+(K)$ and a flow $v_\eta \in \calM^d(K)$. We will show that $v_\eta$ is an irrigation flow sending $\mu$ to $\nu$ and satisfying $\mathbf{M}_\alpha(v_\eta) \leq \mathbf{I}_\alpha(\eta)$ under some hypotheses.

\begin{prop}
If $\eta \in \mathrm{IP}(\mu,\nu)$ then $v_\eta \in \mathrm{IF}(\mu,\nu)$.
\end{prop}
\begin{proof}
Let us calculate the distributional divergence of $v_\eta$. For $\phi \in \calC(K)$, we have
\begin{align*}
\langle \nabla \cdot v_\eta, \phi \rangle = - \langle v_\eta, \nabla\phi \rangle &= -\int_\Gamma \int_0^\infty \nabla\phi(\gamma(t) \cdot \dot{\gamma}(t) \diff t\, \eta(\diff \gamma)\\
&= \int_{\Gamma^1} \left(\phi(\gamma(0)) - \phi(\gamma(\infty))\right) \eta(\diff \gamma)\\
&= \int_K \phi(x) \mu(\diff x) - \int_K \phi(x) \nu(\diff x),
\end{align*}
thus $\nabla \cdot v_\eta = \mu - \nu \in \calM(K)$, which implies that $v_\eta \in \mathrm{IF}(\mu,\nu)$.
\end{proof}

\begin{prop}\label{lag_to_eul}
If $\eta$ is an essentially simple and rectifiable irrigation plan, in particular if $\eta$ is optimal, we have
\[\mathbf{M}_\alpha(v_\eta) \leq \mathbf{I}_\alpha(\eta).\]
\end{prop}
\begin{proof}
By Remark \ref{intens_rectif_rem} we know that $i_\eta = m_\eta \haus^1 = \theta_\eta \haus^1$. Since $\abs*{v_\eta} \leq i_\eta = \theta_\eta \haus^1$, $v_\eta$ has an $\haus^1$-density which is less than $\theta_\eta$. Using the energy formula, we have
\[\mathbf{I}_\alpha(\eta) = \mathbf{E}_\alpha(\eta) = \int_K \theta_\eta^\alpha(x) \haus^1(dx) \geq \int_K \abs*{v_\eta(x)}^\alpha \haus^1(\diff x) = \mathbf{M}_\alpha(v_\eta).\]
\end{proof}

We have therefore proven that we have
\[\inf_{\mathrm{IF}(\mu,\nu)} \quad \mathbf{M}_\alpha\quad  \leq \quad \inf_{\mathrm{IP}(\mu,\nu)} \quad \mathbf{I}_\alpha,\]
and that if $\eta$ is optimal, $v_\eta$ is a good optimal candidate for the Eulerian problem \eqref{eul_pb}.

\subsection{From Eulerian to Lagrangian}

Given an irrigation flow $v \in \mathrm{IF}(\mu,\nu)$ of finite cost $\mathbf{M}_\alpha$, we would like to build an irrigation plan $\eta \in \mathrm{IP}(\mu,\nu)$ such that $v = v_\eta$ (and whose cost is less than $v$). This is not true in general but a Smirnov decomposition gives the result if $v$ is optimal for \eqref{eul_pb}.

\subsubsection*{Cycle} If $v \in \mathrm{IF}(K)$, we say that $w \in \mathrm{IF}(K)$ is a \emph{cycle} of $v$ if $\abs*{v} = \abs*{w} + \abs*{v-w}$ and $\nabla \cdot  w = 0$. It is easy to check that if $v$ is rectifiable then $w$ and $v-w$ are also rectifiable. The following Smirnov decomposition is proved by Santambrogio via a Dacorogna-Moser approach in \cite{santambrogio2014dacorogna}.

\begin{thm}[Irrigation flow decomposition]
Given an irrigation flow $v \in \mathrm{IF}(\mu,\nu)$, there is an irrigation plan $\eta \in \mathrm{IP}(\mu,\nu)$ and a cycle $w \in \mathrm{IF}(K)$ satisfying
\begin{enumerate}[label=(\roman*)]
\item $v = v_\eta + w$,
\item $i_\eta \leq \abs*{v}$.
\end{enumerate}
\end{thm}

\begin{cor}
If $v$ is an optimal irrigation flow in $\mathrm{IF}(\mu,\nu)$, there is an irrigation plan $\eta \in \mathrm{IP}(\mu,\nu)$ such that 
\begin{enumerate}[label=(\roman*)]
\item\label{smirnov_dec1} $v = v_\eta$,
\item\label{smirnov_dec2} $\abs*{v_\eta} = i_\eta$.
\end{enumerate}
\end{cor}
\begin{proof}
Let us take $v_\eta, w$ as in the previous theorem. Since $\mathbf{M}_\alpha(v) < \infty$, $v$ and $v_\eta$ are rectifiable, and by optimality of $v$ one has
\[\int_K \abs*{v_\eta(x)}^\alpha \haus^1(\diff x) = M_\alpha(v_\eta) \geq M_\alpha(v) = \int_K (\abs*{v_\eta(x)} + \abs*{w(x)})^\alpha \haus^1(\diff x),\]
thus we must have $\abs*{w(x)} = 0$ $\haus^1$-a.e., which means $w = 0$, thus $v = v_\eta$ and \ref{smirnov_dec1} holds. This implies $\abs*{v_\eta} \leq i_\eta \leq \abs*{v} = \abs*{v_\eta}$ and thus we have the equality $\abs*{v_\eta} = i_\eta$ wanted in \ref{smirnov_dec2}.
\end{proof}

\begin{prop}\label{eul_to_lag}
If $v$ is an optimal irrigation flow in $\mathrm{IF}(\mu,\nu)$, one can find an irrigation plan $\eta \in \mathrm{IP}(\mu,\nu)$ such that
\[\mathbf{I}_\alpha(\eta) \leq \mathbf{M}_\alpha(v).\]
\end{prop}

\begin{proof}
Take $\eta$ as in the previous corollary. Since $\mathbf{M}_\alpha(v) < \infty$, $v$ is rectifiable and $i_\eta = \abs*{v}$ is concentrated on a rectifiable set, which means by definition that $\eta$ is rectifiable. As a consequence $\abs*{v} = i_\eta = m_\eta \haus^1$ and we have
\begin{align*}
\mathbf{M}_\alpha(v) &= \int_K \abs*{v_\eta(x)}^\alpha \haus^1(\diff x) = \int_K m_\eta^\alpha(x) \haus^1(\diff x),\shortintertext{while}
\mathbf{I}_\alpha(\eta) &= \bar{\mathbf{E}}_\alpha(\eta) = \int_K \theta_\eta^{\alpha -1}(x) m_\eta(x) \haus^1(\diff x).
\end{align*}
We would like $\mathbf{I}_\alpha(\eta) \leq \mathbf{M}_\alpha(v)$, which is a priori not necessarily the case for the $\eta$ we have constructed, since it is not necessarily essentially simple. Instead, take a simple replacement $\tilde{\eta} \in \mathrm{IP}(\mu,\nu)$ satisfying $m_{\tilde{\eta}} = \theta_{\tilde{\eta}} \leq \theta_\eta \leq m_\eta$. Then we get
\[
\mathbf{I}_\alpha(\tilde{\eta}) = \int_K \theta_{\tilde{\eta}}^\alpha \haus^1(\diff x) \leq \int_K m_\eta^\alpha(x) \haus^1(\diff x) = \mathbf{M}_\alpha(v)\]
which yields the result.
\end{proof}
\begin{rem}
Since the minima in the Eulerian and Lagrangian problems are actually the same as we shall see in Theorem \ref{eq_thm}, the previous inequality is an equality, which implies that $\theta_\eta = \theta_{\tilde{\eta}} = m_\eta$ $\haus^1$-a.e. thus $\eta$ was actually optimal hence simple.
\end{rem}

\subsection{The equivalence theorem}

We are now able to formulate the equivalence between the Lagrangian and Eulerian models.

\begin{thm}[Equivalence theorem]\label{eq_thm}
If $1 - \frac{1}{d} < \alpha < 1$ and $\mu,\nu \in \prob(K)$, the Eulerian problem \eqref{eul_pb} and the Lagrangian problem \eqref{lag_pb} are equivalent in the following sense:
\begin{enumerate}[label=(\roman*)]
\item\label{equiv1} the minima are the same
\[\min_{\eta \in \mathrm{IP}(\mu,\nu)} \mathbf{I}_\alpha(\eta) = \min_{v\in \mathrm{IF}(\mu,\nu)} \mathbf{M}_\alpha(v),\]
\item\label{equiv2} if $v$ is optimal in $\mathrm{IF}(\mu,\nu)$, it can be represented by an optimal irrigation plan, \ie $v = v_\eta$ for some optimal $\eta \in \mathrm{IP}(\mu,\nu)$,
\item\label{equiv3} if $\eta$ is optimal in $\mathrm{IP}(\mu,\nu)$, then $v_\eta$ is optimal in $\mathrm{IF}(\mu,\nu)$ and $i_\eta = \abs*{v_\eta}$.
\end{enumerate}
\end{thm}

\begin{proof}
It all follows from Proposition \ref{lag_to_eul} and Proposition \ref{eul_to_lag}. The equality $i_\eta = \abs*{v_\eta}$ comes from
\[\mathbf{I}_\alpha(\eta) = \int_K i_\eta^\alpha(x) \haus^1(\diff x) \geq \int_K \abs*{v_\eta(x)}^\alpha \haus^1(\diff x) = \mathbf{M}_\alpha(v_\eta),\]
since we have equality everywhere by optimality of $v_\eta$ and $\eta$.
\end{proof}

\begin{rem}
Notice in particular that the equality $i_\eta = \abs*{v_\eta}$ implies that curves of $\eta$ have the same tangent vectors when they coincide. To be more precise there is an $\haus^1$-a.e. defined function $\tau : D_\eta \mapsto \Sp^{d-1}$ such that for $\eta$-a.e. $\gamma \in \Gamma$, for $\haus^1$-a.e. $x \in \gamma$, $\dot{\gamma}(t) = \abs*{\dot{\gamma}(t)} \tau(x)$ whenever $\gamma(t) = x$.
\end{rem}

\end{document}